\documentclass[11pt]{article}

\usepackage[a4paper,left=2.5cm,right=2.5cm,top=2.5cm,bottom=2.5cm]{geometry}
\usepackage{amsfonts}
\usepackage{amsmath}
\usepackage{epsfig}
\usepackage{amssymb}
\usepackage{latexsym}
\usepackage{graphicx}
\usepackage{mathrsfs}
\usepackage{color}
\usepackage{euscript}
\usepackage[round]{natbib}

\def\sd{\mathop{\rm sd}\nolimits}

\newtheorem{theorem}{Theorem}[section]
\newtheorem{lemma}[theorem]{Lemma}

\newenvironment{proof}{\noindent {\bf Proof}\ }{$\Box$\bigskip}

\numberwithin{equation}{section}

\newcommand{\R}{\mathbb{R}}

\newcommand{\E}{\mathrm{E}}

\newcommand{\p}{\mathrm{P}}
\newcommand{\eps}{\varepsilon}

\begin{document}
\title{Bayesian recovery of the initial condition for the heat equation\footnote{Appeared in \emph{Communications in Statistics -- Theory and Methods} Volume 42, Issue 7, pp. 1294--1313.}}

\author{B.T. Knapik\thanks{Corresponding author, Department of Mathematics, VU University Amsterdam, {\tt b.t.knapik@vu.nl}\newline  Research supported by Netherlands Organization for Scientific Research NWO},
A.W. van der Vaart\thanks{Mathematical Institute, Leiden University, {\tt avdvaart@math.leidenuniv.nl}} \
and J.H. van Zanten\thanks{Korteweg--de Vries Institute for Mathematics, University of Amsterdam, {\tt j.h.vanzanten@uva.nl}}}

\maketitle

\begin{abstract}
We study a Bayesian approach to recovering the initial condition for the heat equation from noisy observations
of the solution at a later time.
We consider a class of prior distributions indexed by a parameter quantifying  `smoothness' and show
that the corresponding posterior distributions contract around the true parameter at a rate that depends on the smoothness of the true initial condition and the smoothness and scale of the prior. Correct combinations of these characteristics lead to the optimal minimax rate.
One type of priors  leads to a rate-adaptive Bayesian procedure. The frequentist coverage of credible sets is shown to depend on the combination of the prior and true parameter as well, with smoother priors leading to zero coverage and rougher priors to (extremely) conservative results. In the latter case credible sets are much larger than frequentist confidence sets, in that the ratio of
diameters diverges to infinity. The results are numerically illustrated by a simulated data example.
\end{abstract}

\section{Introduction}

 Suppose a differential equation describes the evolution of some feature of a system (e.g., heat conduction), depending on its initial value (at time $t = 0$). We observe the feature at time $T > 0$, in the presence of noise or measurement errors,
and the aim is to recover the initial condition.
Inverse problems of this type are often ill-posed in the sense that the solution operator of the differential equation, which maps the function describing the initial state to the function that describes the state at the later time $T > 0$ at which we observe the system, does typically not have a well-behaved, continuous inverse.
This means that in many cases some form of regularization is necessary to solve the inverse problem and to deal with the noise.

In this paper we study a Bayesian approach to this problem for the particular example of recovering the initial condition
for the heat equation. Specifically, we assume we have noisy observations of the solution $u$ to the
Dirichlet problem for the heat equation
\begin{equation}\label{Heat}
\frac{\partial}{\partial t}u(x,t) = \frac{\partial^2}{\partial x^2} u(x,t), \quad u(x,0) = \mu(x), \quad u(0,t)=u(1,t) = 0,
\end{equation}
where $u$ is defined on $[0,1]\times[0,T]$ and the function $\mu \in L^2[0,1]$ satisfies  $\mu(0)=\mu(1)=0$.  The solution to (\ref{Heat})
is given by
\[
u(x,t) = \sqrt{2}\sum_{i=1}^\infty \mu_i e^{-i^2\pi^2t}\sin(i\pi x),
\]
where $(\mu_i)$ are the coordinates of $\mu$ in the basis $e_i =  \sqrt{2}\sin(i\pi x)$, for $i \geq 1$. In other words,
it holds that $u(\cdot, T) = K\mu$, for $K$ the linear operator on $L^2[0,1]$ that is
diagonalized by the basis $(e_i)$ and that has corresponding eigenvalues
$\kappa_i = \exp({-i^2\pi^2 T})$, for $i \geq 1$.
We assume we observe the solution $K\mu$ in white noise of intensity $1/n$. By expanding in the basis $(e_i)$ this is equivalent to
observing the sequence of noisy, transformed Fourier coefficients
 $Y = (Y_1, Y_2, \ldots)$ satisfying
\begin{equation}\label{Model}
Y_i = \kappa_i \mu_{i} + \frac{1}{\sqrt{n}}Z_i, \qquad i = 1, 2, \ldots,
\end{equation}
for $(\mu_i)$ and $(\kappa_i)$ as above, and $Z_1, Z_2, \ldots$ independent, standard normal random variables.
The aim is to recover the coefficients $\mu_i$, or equivalently, the initial condition $\mu = \sum_{i=1}^\infty \mu_ie_i$, under the assumption that the signal-to-noise ratio tends to infinity (so $n \to \infty$).

This heat conduction inverse problem has been studied in frequentist literature \citep[see, e.g.,][]{Bissantz, Cavalier, Cavalier2, Golubev, Mair2, Mair} and has also been addressed in Bayesian framework (with additional assumptions on the noise),  cf.\ \cite{Stuart}. For more background on how this backward heat conduction problem arises in practical problems, see for instance \cite{Beck} or \cite{Engl}, and the references therein. Since the $\kappa_i$ decay in a sub-Gaussian manner, the estimation of $\mu$ is very hard in general. It is well known for instance that
the minimax rate of estimation for $\mu$ in a Sobolev ball of regularity $\beta$ (see Sec. 1.1) relative to the $\ell^2$-loss is only  $(\log n)^{-\beta/2}$. This rate is attained by various methods, including generalized \emph{Tikhonov regularization} and spectral cut-off \citep{Bissantz, Mair2, Mair, Golubev}.

Convergence rates for Bayesian methods for problems like  (\ref{Model})
have only been studied for the case that $\kappa_i$ decays like a power of $i$, see \cite{KVZ}.
In this paper, like in \cite{KVZ}, we put  product priors of the form
\begin{equation}\label{Prior}
\Pi = \bigotimes_{i=1}^\infty  N(0, \lambda_i)
\end{equation}
on the sequence $(\mu_i)$ and study the corresponding sequence of posterior distributions.
The results we obtain  are different from the ones in \cite{KVZ} in a number of ways however.
First of all, it is in this case not true that to obtain optimal contraction rates for the posterior, we need
to match the regularities of the true sequence $\mu_0$ and the prior exactly. Any degree of oversmoothing
will do as well. Moreover, if the prior variances $\lambda_i$ are
chosen sub-Gaussian, then we obtain the optimal rate $(\log n)^{-\beta/2}$
for {\em any} $\beta$-regular $\mu_0$, i.e.,\ we obtain a rate-adaptive procedure.
Unfortunately however, these very smooth priors behave badly from another point of view. We show that asymptotically, the frequentist coverage
of credible sets based on these priors is $0$ for a very large class of true $\mu_0$'s. As in \cite{KVZ}
we see that asymptotic coverage $1$ is obtained when the prior is less  regular than the truth.
The radius of a credible set is in that case however of a strictly larger order than the radius of the corresponding frequentist
credible set, which is another difference with the findings in \cite{KVZ} for polynomial $\kappa_i$.

These statements are made precise and are refined to include the possibility of rescaling the priors in Sec. 2.
On a qualitative level, the conclusion of the results must be that in the severely ill-posed case that we study in this paper
it is advisable to use a prior that is slightly less regular than the truth, just as in the mildly ill-posed case of \cite{KVZ}.
 Unfortunately, the
corresponding Bayesian credible sets can be  very large in the present setting and hence of  limited use.
The results in Sec.~2 all deal with the recovery of the full parameter $\mu$.
In Sec.~3 we derive the analogous results for the problem of estimating
 linear functionals of $\mu$. The results are numerically illustrated in Sec.~4. Sec.~5 contains proofs of the results presented in Secs.~2 and 3. Auxiliary lemmas are presented in Sec.~6.

\subsection{Notation}

\noindent For $\beta > 0$, the Sobolev norm  $\|\mu\|_\beta$ and the $\ell^2$-norm $\|\mu\|$ of
an element $\mu \in \ell^2$ are defined by
\[
\|\mu\|_\beta^2 = \sum_{i=1}^\infty \mu_i^2i^{2\beta}, \qquad \|\mu\|^2 = \sum_{i=1}^\infty \mu_i^2,
\]
and the corresponding Sobolev space by $S^\beta = \{\mu \in \ell^2: \|\mu\|_\beta < \infty\}$.

For two sequences $(a_n)$ and $(b_n)$ of numbers,  $a_n \asymp b_n$ means that $|a_n/b_n|$ is bounded away from zero and infinity as $n \to \infty$, $a_n \lesssim b_n$ means that $a_n/b_n$ is bounded, $a_n \sim b_n$ means that $a_n/b_n \to 1$ as $n \to \infty$, and $a_n \ll b_n$ means that $a_n / b_n \to 0$ as $n \to \infty$. For two real numbers $a$ and $b$, we denote by $a \vee b$ their maximum, and by $a \wedge b$ their minimum.

\section{Recovering the full parameter}

\noindent Under the model \eqref{Model} and the prior \eqref{Prior} the coordinates $(\mu_{0,i}, Y_i)$ of the vector $(\mu_0, Y)$ are independent, and hence the conditional distribution of $\mu_0$ given $Y$ factorizes over the coordinates as well. Thus the computation of the posterior distribution reduces to countably many posterior computations in conjugate normal models. It is straightforward to verify that the posterior distribution $\Pi_n(\,\cdot \mid Y)$ is given by
\begin{equation}\label{PosteriorFull}
\Pi_n(\,\cdot \mid Y) = \bigotimes_{i=1}^\infty  N\biggl(\frac{n\lambda_i\kappa_i}{1+n\lambda_i\kappa_i^2}Y_i,\frac{\lambda_i}{1+n\lambda_i\kappa_i^2}\biggr).
\end{equation}

Our first theorem shows that the posterior contracts as $n \to \infty$ to the true parameter at a rate $\eps_n$
and quantifies how this rate  depends on the behavior of the sequence $(\lambda_i)$ of prior variances  and the
regularity  $\beta$ of the true parameter $\mu_0$. We say the posterior {\em contracts around $\mu_0$ at the rate
$\eps_n$} if
\[
\E_{\mu_0}\Pi_n(\mu:\|\mu-\mu_0\|\geq M_n\eps_n\mid Y)\to 0
\]
 for every $M_n \to \infty$, where the expectation is under the true model governed by the parameter $\mu_0$.

\begin{theorem}\label{RateFullS}
Suppose the true parameter $\mu_0$ belongs to $S^\beta$ for $\beta > 0$.

If $\lambda_i = \tau_n^2i^{-1-2\alpha}$ for some $\alpha > 0$ and $\tau_n >0$ such that $n\tau_n^2\to \infty$, then
the posterior contracts around $\mu_0$ at the rate
\begin{equation}\label{RateF1}
\eps_n = (\log  n\tau_n^2)^{-\beta/2} +  \tau_n(\log n\tau_n^2)^{-\alpha/2}.
\end{equation}
The rate is uniform over $\mu_0$ in balls in $S^\beta$. In particular:
\begin{itemize}
\item[(i)] If $\tau_n \equiv 1$, then $\eps_n = (\log n)^{-(\beta\wedge \alpha)/2}$.
\item[(ii)] If $n^{-1/2+\delta} \lesssim \tau_n \lesssim (\log n)^{(\alpha-\beta)/2}$, for some $\delta > 0$, then $\eps_n = (\log n)^{-\beta/2}$.
\end{itemize}

If $\lambda_i = e^{-\alpha i^2}$ for some $\alpha > 0$ then
the posterior contracts around $\mu_0$ at the rate
\begin{equation}\label{RateF2}
\eps_n = (\log n)^{-\beta/2}.
\end{equation}
The rate is uniform over $\mu_0$ in balls in $S^\beta$.
\end{theorem}

We think of the  parameters $\beta$ and $\alpha$ as the regularity of the true parameter $\mu_0$ and the prior, respectively. The first is validated by the fact that in the heat equation case $(e_i)$ is the (sine) Fourier basis of $L^2[0,1]$. Therefore $\beta$
quantifies the smoothness of $\mu_0$ in Sobolev sense. In case of the polynomial decay of the variances of the prior (later referred to as the polynomial prior), the parameter $\alpha$ is also closely related to  Sobolev regularity.

The minimax rate of convergence over a Sobolev ball $S^\beta$ is of the order $(\log n)^{-\beta/2}$.
Now consider the case  $\lambda_i = \tau_n^2i^{-1-2\alpha}$. By statement (i) of the theorem the posterior
contracts at the optimal minimax rate  if the regularity of the prior is at least the regularity of the truth ($\alpha \geq \beta$) and the scale $\tau_n$ is fixed. Alternatively, the optimal rate is also attained by appropriately scaling a prior of any regularity. Note that if $\alpha \geq \beta$ scaling is redundant.
The theorem shows that `correct' specification of the prior regularity gives the optimal rate. In contrast to \cite{KVZ}
however, the regularity of the prior does not have to match exactly the regularity of the truth. Moreover, even though rough priors still need to be scaled to give the optimal rate, there is no restriction on the `roughness'.

The second assertion of the theorem shows that for very smooth priors (where we take $\lambda_i = e^{-\alpha i^2}$) the contraction rate is always optimal. Since the prior does not depend on the unknown regularity $\beta$,
the procedure is  \emph{rate-adaptive} in this case.

Both choices of priors lead to the conclusion that oversmoothing yields the optimal rate, and this has been noted also in the frequentist literature \citep[see][]{Mair2}. A fully adaptive frequentist method is presented in  \cite{Bissantz}, and in both situations
the optimal performance is caused by the dominating  bias. However, in Bayesian inference one often takes the spread in the posterior distribution as a quantification of uncertainty. If $\lambda_i = e^{-\alpha i^2}$ this spread is much smaller than the minimax rate.
To understand the implications,  we next consider the frequentist coverage of credible sets. As the posterior is Gaussian, it is natural to center a credible region at the posterior mean. Different shapes of such a set could be considered, but the natural counterpart of the preceding theorem is to consider balls. The study of linear functionals in the next section makes it possible to consider pointwise credible bands as well.

A credible ball centered at the posterior mean $\hat\mu$, where $\hat \mu_i = n\lambda_i\kappa_i(1+n\lambda_i\kappa_i^2)^{-1}Y_i$, takes the form
\begin{equation}\label{CredibleBall}
\hat\mu + B(r_{n,\gamma}):= \bigl\{\mu \in \ell^2 : \|\mu-\hat\mu\| < r_{n,\gamma}\bigr\},
\end{equation}
where $B(r)$ denotes an $\ell^2$-ball of radius $r$ around $0$
and  the radius $r_{n,\gamma}$ is determined such that
\begin{equation}\label{Credibility}
\Pi_n\bigl(\hat\mu + B(r_{n,\gamma})\mid Y\bigr) = 1-\gamma.
\end{equation}
Because the spread of the posterior is not dependent on the data, neither is the radius $r_{n,\gamma}$. The frequentist \emph{coverage} or confidence of the set \eqref{CredibleBall} is, by definition,
\begin{equation}\label{AsymptoticCov}
\p_{\mu_0}\bigl(\mu_0\in \hat\mu + B(r_{n,\gamma})\bigr),
\end{equation}
where under the probability measure $\p_{\mu_0}$ the variable $Y$ follows \eqref{Model} with $\mu = \mu_0$. We shall consider the coverage as $n \to \infty$ for fixed $\mu_0$, uniformly in Sobolev balls, and also along sequences $\mu_0^n$ that change with $n$.

The following theorem shows that the relation of the coverage to the credibility level $1-\gamma$ is mediated by
the regularity of the true $\mu_0$ and the two parameters controlling the regularity of the prior---$\alpha$ and the scaling $\tau_n$---for both types of priors. For further insight, the credible region is also compared to the `correct' frequentist confidence ball $\hat\mu +B(\tilde r_{n,\gamma})$ chosen so that the probability in~\eqref{AsymptoticCov} is exactly equal to $1-\gamma$.

\begin{theorem}\label{CredibilityFull}
Suppose the true parameter $\mu_0$ belongs to $S^\beta$ for $\beta > 0$.

If  $\lambda_i = \tau_n^2i^{-1-2\alpha}$ for some $\alpha > 0$ and $\tau_n >0$ such that $n\tau_n^2\to \infty$, then asymptotic coverage of the credible region~\eqref{CredibleBall} is
\begin{itemize}
\item[(i)] 1, uniformly in $\mu_0$ with $\|\mu_0\|_\beta \leq 1$, if $\tau_n \gg (\log n)^{(\alpha - \beta)/2}$; in this case $r_{n,\gamma} /\tilde r_{n,\gamma} \to \infty$.
\item[(ii)] 1, uniformly in $\mu_0$ with $\|\mu_0\|_\beta \leq r$ for $r$ small enough, if $\tau_n \asymp (\log n)^{(\alpha - \beta)/2}$;\\
1, for every fixed $\mu_0 \in S^\beta$, if $\tau_n \asymp (\log n)^{(\alpha - \beta)/2}$.
\item [(iii)] 0, along some $\mu_0^n$ with $\sup_n \bigl\|\mu_0^n\bigr\|_\beta <\infty$, if $\tau_n \lesssim (\log n)^{(\alpha - \beta)/2}$.
\end{itemize}

If $\lambda_i = e^{-\alpha i^2}$ for some $\alpha > 0$, then the asymptotic coverage of the credible region~\eqref{CredibleBall} is
\begin{itemize}
\item[(iv)] 0, for every $\mu_0$ such that $|\mu_{0,i}| \gtrsim e^{-ci^2/2}$ for some $c < \alpha$.
\end{itemize}
If $\tau_n \equiv 1$, then the cases (i), (ii), and (iii) arise if $\alpha < \beta$, $\alpha = \beta$ and $\alpha \geq \beta$, respectively. If $\alpha > \beta$ in case $(iii)$ the sequence $\mu_0^n$ can then be chosen fixed.
\end{theorem}

The easiest interpretation of the theorem is in the situation without scaling $(\tau_n\equiv 1)$. Then oversmoothing the prior (case (iii): polynomial prior with $\alpha > \beta$, and case (iv): exponential prior) has disastrous consequences for the coverage of the credible sets, whereas undersmoothing (case (i): polynomial prior with $\alpha < \beta$) leads to (very) conservative sets. Choosing a prior of correct regularity (case (ii) and (iii): polynomial prior with $\alpha = \beta$) gives mixed results, depending on the norm of the true $\mu_0$. These conclusions are analogous to the ones that can be drawn from Theorem 4.2 in \cite{KVZ} for the mildly ill-posed case.

There is one crucial difference, namely the radius of the conservative sets in case (i) are not of the correct order of magnitude. It means that the radius $\tilde r_{n,\gamma}$ of the `correct' frequentist confidence ball is of strictly smaller order than the radius of the Bayesian credible ball.

By Theorem~\ref{RateFullS} the optimal contraction rate is obtained by smooth priors. Combining the two theorems leads to the conclusion that polynomial priors that slightly undersmooth the truth might be preferable. They attain a nearly optimal rate of contraction and the spread of their posterior gives a reasonable sense of uncertainty. Slightly undersmoothing is only possible however if an assumption about the regularity of the unknown true function is made. It is an important open problem to devise methods that achieve this automatically, without knowledge about the true regularity. Exponential priors, although adaptive and rate-optimal,
often lead to very bad pointwise credible bands.
%

\section{Recovering linear functionals of the parameter}

\noindent In this section we consider the posterior distribution of a linear functional $L\mu$ of the parameter. In the Bayesian setting we consider \emph{measurable linear functionals relative to the prior}, covering the class of continuous functionals, but also certain discontinuous functionals (for instance point evaluation), following the definition of \cite{Skorohod}. Let $(l_i) \in \R^\infty$ satisfy $\sum_{i=1}^\infty l_i^2\lambda_i < \infty$. Then it can be shown that $L\mu := \lim_{n\to \infty} \sum_{i=1}^n l_i\mu_i$ exists for all $\mu = (\mu_i)$ in a (measurable) subspace of $\ell^2$ with $\bigotimes_{i=1}^\infty  N(0, \lambda_i)$-probability one. We define $L\mu = 0$ if the limit does not exist.

The posterior of the linear functional $L\mu$ can be obtained from \eqref{PosteriorFull} and the definition given above \citep[see also][]{KVZ}
\begin{equation}\label{PosteriorMargin}
\Pi_n(\mu: L\mu \in \cdot\mid Y) = N\biggl(\sum_{i=1}^\infty \frac{nl_i\lambda_i\kappa_i}{1+n\lambda_i\kappa_i^2}Y_i,
\sum_{i=1}^\infty \frac{l_i^2\lambda_i}{1+n\lambda_i\kappa_i^2}\biggr).
\end{equation}
We measure the smoothness of the functional $L$ by the size of the coefficients $l_i$, as $i \to \infty$. It is natural to assume that the sequence $(l_i)$ is in the Sobolev space $S^q$ for some $q$, but also more controlled behavior will be assumed in following theorems.
We say that the {\em marginal posterior of $L\mu$ contracts around $L\mu_0$ at the rate $\eps_n$} if
\[
\E_{\mu_0}\Pi_n(\mu:|L\mu-L\mu_0|\geq M_n\eps_n\mid Y)\to 0
\]
as $n \to \infty$, for every sequence $M_n \to \infty$.

\begin{theorem}\label{RateFunctS}
Suppose the true parameter $\mu_0$ belongs to $S^\beta$ for $\beta > 0$.

If $\lambda_i = \tau_n^2 i^{-1-2\alpha}$ for some $\alpha > 0$ and $\tau_n > 0$ such that $n\tau_n^2\to \infty$, and the representer $(l_i)$ of the linear functional $L$ is contained in $S^q$, or $|l_i| \lesssim i^{-q-1/2}$ for some $q \geq -\beta$,
then the marginal posterior of $L\mu$ contracts around $L\mu_0$ at the rate
\begin{equation}\label{RateL1}
\eps_n = (\log n\tau_n^2)^{-(\beta+q)/2} +  \tau_n(\log n\tau_n^2)^{-(1/2+\alpha+q)/2}.
\end{equation}
The rate is uniform over $\mu_0$ in balls in $S^\beta$. In particular:
\begin{itemize}
\item[(i)] If $\tau_n \equiv 1$, then $\eps_n = (\log n)^{-(\beta\wedge (1/2+\alpha)+q)/2}$.
\item[(ii)] If $n^{-1/2+\delta} \lesssim \tau_n \lesssim (\log n)^{(1/2+\alpha-\beta)/2}$, for some $\delta > 0$,  then $\eps_n = (\log n)^{-(\beta+q)/2}$.
\end{itemize}

If $\lambda_i= e^{-\alpha i^2}$ for some $\alpha > 0$ then the marginal posterior of $L\mu$ contracts around $L\mu_0$ at the rate
\begin{equation}\label{RateL2}
\eps_n = (\log n)^{-(\beta+q)/2}.
\end{equation}
The rate is uniform over $\mu_0$ in balls in $S^\beta$.
\end{theorem}

The minimax rate over a ball in the Sobolev space $S^\beta$ is known to be bounded above by $(\log n)^{-(\beta+q)/2}$ (for the case of $q = -1/2$ see \citealp{GoldDeconv}, and for general $q$ in a closely related model see \citealp{Butucea}). In view of Theorem~\ref{RateFullS}, it is not surprising that exponential priors yield this optimal rate. In case of polynomial prior this rate is attained without scaling if and only if the prior smoothness $\alpha$ is greater than or equal to $\beta$ minus 1/2. Here we observe a similar phenomenon as in \cite{KVZ}, where the `loss' in smoothness by $1/2$ is discussed. The regularity of the parameter in the Sobolev scale is not the appropriate type of regularity to consider for estimating a linear functional $L\mu$.
If the polynomial prior is too rough, then the minimax rate may still be attained by scaling the prior. The upper bound on the scaling is the same as in the global case (see Theorem~\ref{RateFullS}.(ii)) \emph{after} decreasing $\beta$ by 1/2. So the `loss in regularity' persists in the scaling.

Because the posterior distribution for the linear functional $L\mu$ is the one-dimensional normal distribution $N(\widehat{L\mu}, s_n^2)$,
where $s_n^2$ is the posterior variance in (\ref{PosteriorMargin}),
the natural \emph{credible interval} for $L\mu$ has endpoints $\widehat{L\mu} \pm z_{\gamma/2}s_n$, for $z_\gamma$ the (lower) standard normal $\gamma$-quantile. The \emph{coverage} of this interval is
\[
\p_{\mu_0}\bigl(\widehat{L\mu} + z_{\gamma/2}s_n \leq L\mu_0 \leq \widehat{L\mu} - z_{\gamma/2}s_n\bigr),
\]
where $Y$ follows \eqref{Model} with $\mu = \mu_0$. In the following theorem we restrict $(l_i)$ to sequences that behave polynomially.

\begin{theorem}\label{CredibilityLin}
Suppose the true parameter $\mu_0$ belongs to $S^\beta$ for $\beta > 0$. Let $\tilde\tau_n = (\log n)^{(1/2+\alpha-\beta)/2}$.

If  $\lambda_i = \tau_n^2i^{-1-2\alpha}$ for some $\alpha > 0$ and $\tau_n > 0$ such that $n\tau_n^2 \to \infty$, and
$|l_i| \asymp i^{-q-1/2}$, then the asymptotic coverage of the interval $\widehat{L\mu} \pm z_{\gamma/2}s_n$ is:
\begin{itemize}
\item[(i)] 1, uniformly in $\mu_0$ such that $\|\mu_0\|_\beta\leq 1$ if $\tau_n \gg \tilde\tau_n$,
\item[(ii)] 1, uniformly in $\mu_0$ with $\|\mu_0\|_\beta \leq r$ for $r$ small enough, if $\tau_n \asymp \tilde\tau_n$;\\
1, for every fixed $\mu_0 \in S^\beta$, if $\tau_n \asymp \tilde\tau_n$,
\item[(iii)] 0, along some $\mu_0^n$ with $\sup_n \bigl\|\mu_0^n\bigr\|_\beta < \infty$, if $\tau_n \lesssim \tilde\tau_n$.
\end{itemize}

If $\lambda_i = e^{-\alpha i^2}$ for some $\alpha > 0$, then the  asymptotic coverage of the interval $\widehat{L\mu} \pm z_{\gamma/2}s_n$ is:
\begin{itemize}
\item[(iv)] 0, for every $\mu_0$ such that $\mu_{0,i}l_i\gtrsim e^{-ci^2/2}i^{-q-1/2}$ for some $c < \alpha$.
\end{itemize}
In case (iii) the sequence $\mu_0^n$ can be taken a fixed element $\mu_0$ in $S^\beta$ if $\tau_n \leq \tilde\tau_n (\log n)^{-\delta}$ for some $\delta > 0$. Furthermore, if $\tau_n\equiv 1$, then the cases (i), (ii) and (iii) arise if $\alpha < \beta-1/2$, $\alpha = \beta-1/2$ and $\alpha \geq \beta-1/2$, respectively. If $\alpha > \beta-1/2$ in case (iii) the sequence $\mu_0^n$ can then be chosen fixed.
\end{theorem}

Similarly as in the problem of full recovery of the parameter $\mu$ oversmoothing leads to coverage 0, while undersmoothing gives (extremely) conservative intervals. In the case of a polynomial prior without scaling the cut-off for under- or oversmoothing is at $\alpha = \beta-1/2$, while the cut-off for scaling is at the optimal rate $\tilde\tau_n$. Exponential priors are bad even for very smooth $\mu_0$, and the asymptotic coverage in this case is always 0. It should be noted that too much undersmoothing is also undesirable, as it leads to very wide credible intervals, and may cause that $\sum_{i=1}^\infty  l_i^2\lambda_i$ is no longer finite.

In contrast with the analogous theorem in \cite{KVZ}, the conservativeness in case of undersmoothing is extreme, as the coverage is 1. Since it holds for every linear functional that can be considered in this setting, we do not have a Bernstein--von Mises theorem. The linear functionals considered in this section are not smooth enough to cancel the ill-posedness of the problem \citep[cf. discussion after Theorem 5.4 in][]{KVZ}.

\section{Simulation example}

\noindent To illustrate our results with simulated data we fix a time $T = 0.1$ and a true function $\mu_0$, which we expand as $\mu_0 = \sum_{i=1}^\infty  \mu_{0,i}e_i$ in the basis $(e_i)$. The simulated data are the noisy and transformed coefficients
\[
Y_i = \kappa_i\mu_{0,i}+\frac{1}{\sqrt{n}}Z_i.
\]
The (marginal) posterior distribution for the function $\mu$ at a point $x$ is obtained by expanding $\mu(x) = \sum_{i=1}^\infty  \mu_i e_i(x)$, and applying the framework of linear functionals $L\mu = \sum_{i=1}^\infty  l_i\mu_i$ with $l_i = e_i(x)$ (so $l_i \lesssim 1$ and $q = -1/2$). Recall
\[
\mu(x)\mid Y \sim N\biggl(\sum_{i=1}^\infty  \frac{n\lambda_i\kappa_ie_i(x)}{1+n\lambda_i\kappa_i^2}Y_i, \sum_{i=1}^\infty  \frac{e_i(x)^2\lambda_i}{1+n\lambda_i\kappa_i^2}\biggr).
\]
We obtained (marginal) posterior pointwise credible bands by computing for every $x$ a central 95\% interval for the normal distribution on the right side of the above display. We considered both types of priors.

Figure~1 illustrates these bands for $n=10^4$ and the polynomial prior. In every of 10 panels in the figure the black curve represents the function $\mu_0$, defined by
\begin{equation}
\label{True}
\mu_0(x) = 4x(x-1)(8x-5), \qquad \mu_{0,i} =\frac{8\sqrt{2}(13+11(-1)^i)}{\pi^3i^3},
\end{equation}
where $\mu_{0,i}$ are the coefficients relative to $e_i$, thus $\mu_0 \in S^\beta$ for every $\beta < 2.5$. The 10 panels represent 10 independent realizations of the data, yielding 10 different realizations of the posterior mean (the red curves) and the posterior pointwise credible bands (the green curves). In the left five panels the prior is given by $\lambda_i = i^{-1-2\alpha}$ with $\alpha = 1$, whereas in the right panels the prior corresponds to $\alpha = 3$. Each of the 10 panels also shows 20 realizations from the posterior distribution. This is also valid for Figure~2, with the exponential prior, so $\lambda_i = e^{-\alpha i^2}$. In the left panels $\alpha = 1$, and in the right panels $\alpha = 5$.

\begin{figure}
\centerline{\includegraphics[width=9cm]{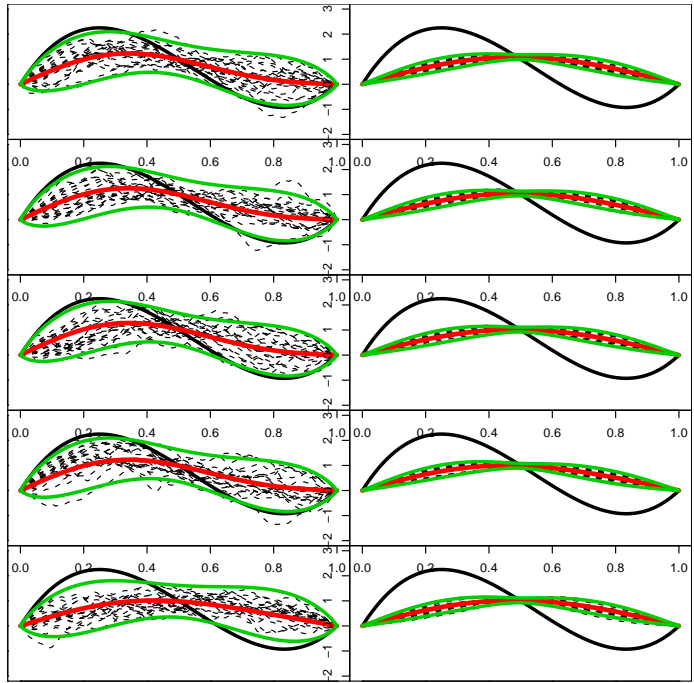}}
\caption{Polynomial prior. Realizations of the posterior mean (red) and (marginal)
posterior credible bands (green), and 20 draws from
the posterior (dashed curves). In all ten panels $n=10^4$.
Left 5 panels: $\alpha=1$; right 5 panels: $\alpha=3$. True curve
(black) given by \eqref{True}.}
\end{figure}

\begin{figure}
\centerline{\includegraphics[width=9cm]{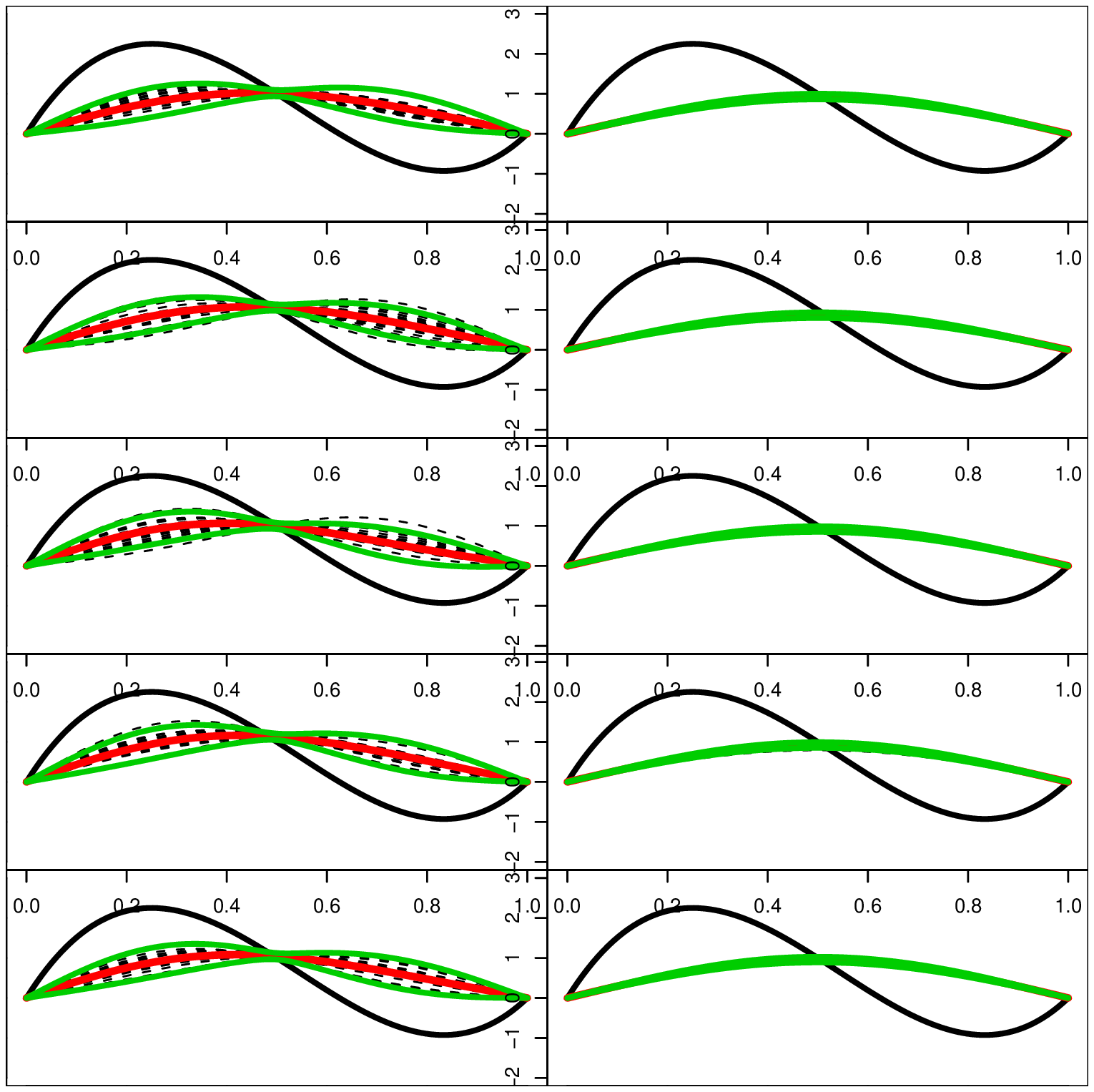}}
\caption{Exponential prior. Realizations of the posterior mean (red) and (marginal)
posterior credible bands (green), and 20 draws from
the posterior (dashed curves). In all ten panels $n=10^4$.
Left 5 panels: $\alpha=1$; right 5 panels: $\alpha=5$. True curve
(black) given by \eqref{True}.}
\end{figure}

\begin{figure}
\centerline{\includegraphics[width=9cm]{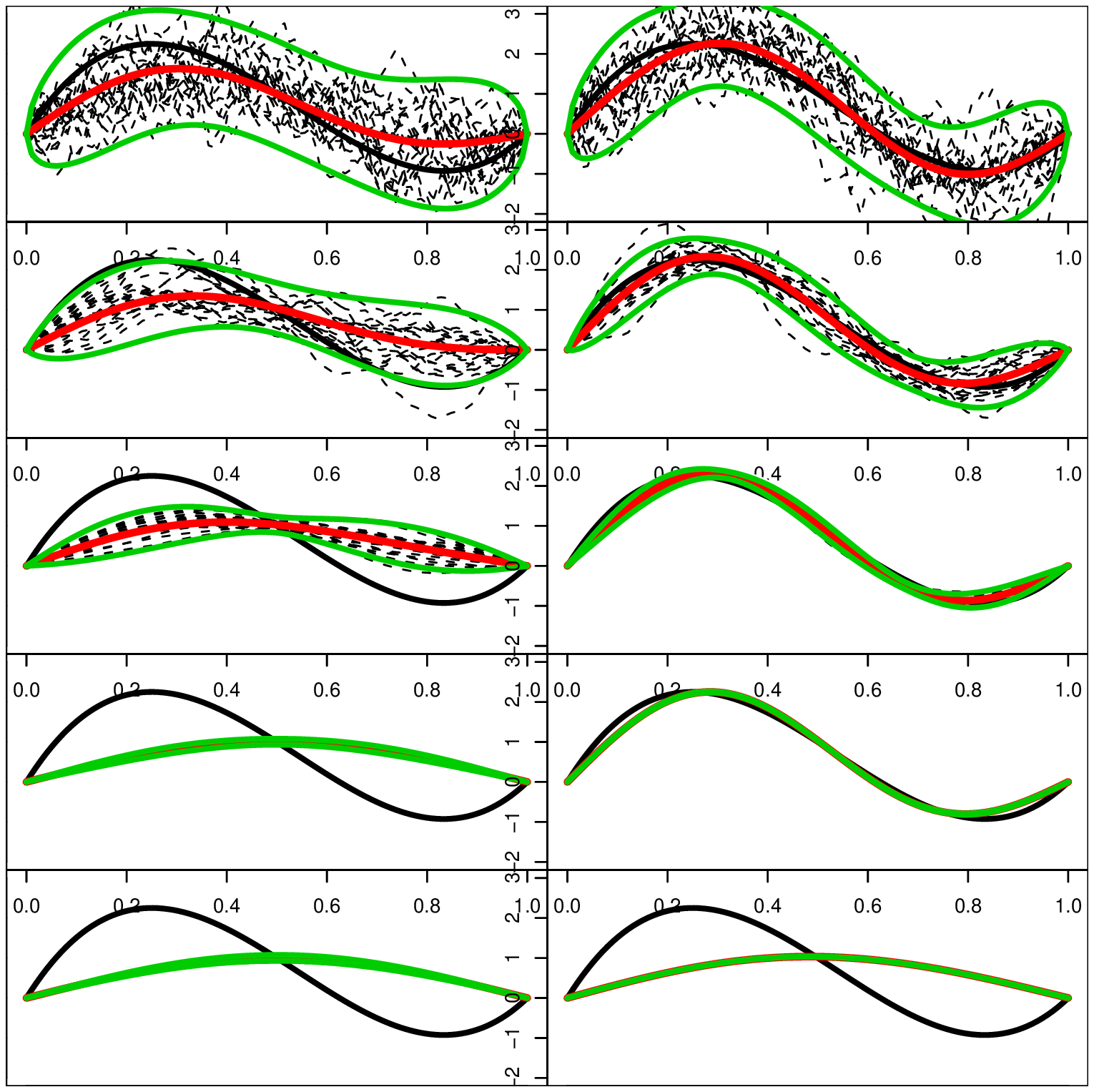}}
\caption{Polynomial prior. Realizations of the posterior mean (red) and (marginal)
posterior credible bands (green), and 20 draws from
the posterior (dashed curves).
Left 5 panels: $n=10^4$ and $\alpha=0.5,1,2,5,10$ (top to bottom);
right 5 panels: $n=10^8$ and $\alpha=0.5,1,2,5,10$ (top to bottom). True curve
(black) given by \eqref{True}.}
\end{figure}

\begin{figure}
\centerline{\includegraphics[width=9cm]{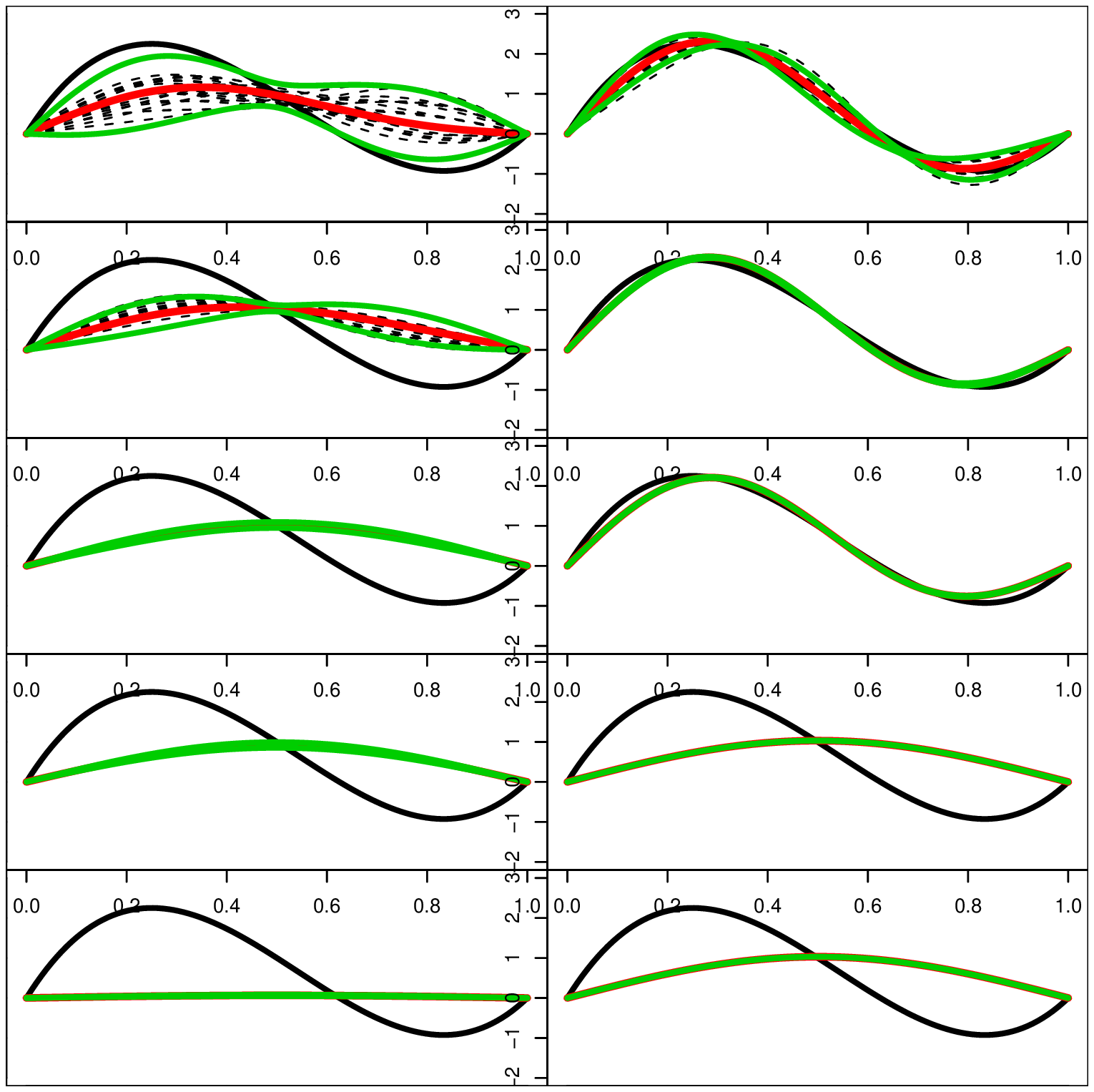}}
\caption{Exponential prior. Realizations of the posterior mean (red) and (marginal)
posterior credible bands (green), and 20 draws from
the posterior (dashed curves).
Left 5 panels: $n=10^4$ and $\alpha=0.5,1,2,5,10$ (top to bottom);
right 5 panels: $n=10^8$ and $\alpha=0.5,1,2,5,10$ (top to bottom). True curve
(black) given by \eqref{True}.}
\end{figure}

A comparison of the left and right panels in Figure~1 shows that the rough polynomial prior ($\alpha = 1$) is aware of the difficulty of inverse problem: it produces wide pointwise credible bands that in (almost) all cases contain nearly the whole true curve. Figure~1 together with Figure~2 show that smooth priors (polynomial with $\alpha = 3$ and both exponential priors) are overconfident: the spread of the posterior distribution poorly reflects the imprecision of estimation. Our theoretical results show that the inaccurate quantification of the estimation error (by the posterior spread) remains even as $n \to \infty$.

The reconstruction, by the posterior mean or any other posterior quantiles, will eventually converge to the true curve.
The specification of the prior influences the speed of this convergence. This is illustrated in Figures~3~and~4. Every of 10 panels in each of the figures is similarly constructed as before, but now with $n=10^4$ and $n=10^8$ for the five panels on the left and right side, respectively, and with $\alpha = 1/2, 1, 2, 5, 10$ for the five panels from top to bottom ($\lambda_i = i^{-1-2\alpha}$ in Figure~3, and $\lambda_i = e^{-\alpha i^2}$ in Figure~4). As discussed above, all exponential priors give the optimal rate, but lead to bad pointwise credible bands. Also smooth polynomial priors give the optimal rate. This can be seen in Figure~3 for $n=10^8$ and $\alpha = 2$ or $5$, where pointwise credible bands are very close to the true curve. However, for $\alpha = 5$ it should be noted that the true curve is mostly outside the pointwise credible band.

\section{Proofs}\label{Proofs}

\subsection{Proof of Theorem~\ref{RateFullS}}

\noindent Let $s_{i,n}$ and $t_{i,n}$ be such that the posterior distribution in \eqref{PosteriorFull} can be denoted by $\bigotimes_{i=1}^\infty  N\bigl(\sqrt{nt_{i,n}} Y_i,s_{i,n}\bigr)$. Because the posterior is Gaussian, it follows that
\begin{equation}\label{NormSquared}
\int\|\mu - \mu_0\|^2\, d\Pi_n(\mu \mid Y) = \|\hat\mu-\mu_0\|^2+\sum_{i=1}^\infty  s_{i,n},
\end{equation}
where $Y$ follows \eqref{Model} with $\mu = \mu_0$, and
\[
\hat\mu = \biggl(\frac{n\lambda_i\kappa_i}{1+n\lambda_i\kappa_i^2}Y_i\biggr)_i = \biggl(\frac{n\lambda_i\kappa_i^2\mu_{0,i}}{1+n\lambda_i\kappa_i^2} + \frac{\sqrt{n}\lambda_i\kappa_iZ_i}{1+n\lambda_i\kappa_i^2}\biggr)_i =: \E_{\mu_0}\hat\mu + \bigl(\sqrt{t_{i,n}}Z_i\bigr)_i.
\]
By Markov's inequality the left side of~\eqref{NormSquared} is an upper bound to $M_n^2\eps_n^2\Pi_n\bigl(\mu:\|\mu -\mu_0\|\geq M_n\eps_n\mid Y)$. Therefore, it suffices to show that the expectation under $\mu_0$ of the right side of the display is bounded by a multiple of $\eps_n^2$. The expectation of the first term is the mean square error of the posterior mean $\hat\mu$, and can be written as the sum $\|\E_{\mu_0}\hat\mu-\mu_0\|^2 + \sum_{i=1}^\infty  t_{i,n}$ of its square bias and `variance'. The second term $\sum_{i=1}^\infty  s_{i,n}$ is deterministic. If $\lambda_i = \tau_n^2i^{-1-2\alpha}$ the three quantities are given by:
\begin{align}
\label{SqBias} \|\E_{\mu_0}\hat\mu-\mu_0\|^2 &= \sum_{i=1}^\infty  \frac{\mu_{0,i}^2}{(1+n\lambda_i\kappa_i^2)^2} = \sum_{i=1}^\infty  \frac{\mu_{0,i}^2}{(1+n\tau_n^2i^{-1-2\alpha}e^{-2\pi^2 T i^2})^2}\\
\label{Var} \sum_{i=1}^\infty  t_{i,n} &= \sum_{i=1}^\infty  \frac{n\lambda_i^2\kappa_i^2}{(1+n\lambda_i\kappa_i^2)^2} = \sum_{i=1}^\infty  \frac{n\tau_n^4i^{-2-4\alpha}e^{-2\pi^2 T i^2}}{(1+n\tau_n^2i^{-1-2\alpha}e^{-2\pi^2 T i^2})^2}\\
\label{PostSpr} \sum_{i=1}^\infty  s_{i,n}&= \sum_{i=1}^\infty  \frac{\lambda_i}{1+n\lambda_i\kappa_i^2} = \sum_{i=1}^\infty  \frac{\tau_n^2i^{-1-2\alpha}}{1+n\tau_n^2i^{-1-2\alpha}e^{-2\pi^2 T i^2}}.
\end{align}

By Lemma~\ref{Norm} (applied with $q=\beta$, $t=0$, $r=0$, $u = 1+2\alpha$, $p = 2\pi^2T$, $v =2$, and $N = n\tau_n^2$) the first term can be bounded by $\log(n\tau_n^2)^{-\beta}$, which accounts for the first term in the definition of $\eps_n$ in \eqref{RateF1}. By Lemma~\ref{Series} (applied with $t=2+4\alpha$, $r= 2\pi^2T$, $u = 1+2\alpha$, $p=2\pi^2T$, $v=2$, and $N=n\tau_n^2$) the second expression is of the order $\tau_n^2(\log n\tau_n^2)^{-1/2-\alpha}$. The third expression is of the order the square of the second term in the definition of $\eps_n$ in \eqref{RateF1}, by Lemma~\ref{Series} (applied with $t=1+2\alpha$, $r= 0$, $u = 1+2\alpha$, $p=2\pi^2T$, $v=1$, and $N=n\tau_n^2$).

The consequences (i)--(ii) follow by verification after substitution of $\tau_n$ as given.

In case of $\lambda_i =e^{-\alpha i^2}$, we replace $i^{-1-2\alpha}$ by $e^{-\alpha i^2}$ and set $\tau_n\equiv 1$ in \eqref{SqBias}--\eqref{PostSpr}. We then apply Lemma~\ref{Norm} (with $q=\beta$, $t=0$, $r=0$, $u = 0$, $p = 2\pi^2T+\alpha$, $v =2$, and $N = n$) and see that the first term can be bounded by $(\log n)^{-\beta}$, which accounts for the first term in the definition of $\eps_n$ in \eqref{RateF2}. By Lemma~\ref{Series} (applied with $t=0$, $r= 2\alpha+2\pi^2T$, $u = 0$, $p=2\pi^2\alpha$, $v=2$, and $N=n$), and again Lemma~\ref{Series} (applied with $t=0$, $r= \alpha$, $u = 0$, $p=\alpha+2\pi^2T$, $v=1$, and $N=n$) the latter two are of the order $n^{-\alpha/(\alpha+2\pi^2T)}$.

\subsection{Proof of Theorem~\ref{CredibilityFull}}

\noindent Because the posterior distribution is $\bigotimes_{i=1}^\infty  N(\sqrt{n t_{i,n}}Y_i,s_{i,n})$, by~\eqref{PosteriorFull}, the radius $r_{n,\gamma}$ in \eqref{Credibility} satisfies $\p(U_n < r_{n,\gamma}^2) = 1-\gamma$, for $U_n$ a random variable distributed as the square norm of an $\bigotimes_{i=1}^\infty  N(0, s_{i,n})$-variable. Under \eqref{Model} the variable $\hat\mu$ is $\bigotimes_{i=1}^\infty  N\bigl((\E_{\mu_0}\hat\mu)_i, t_{i,n}\bigr)$-distributed, and thus the coverage \eqref{AsymptoticCov} can be written as
\begin{equation}\label{eq: cov}
\p\bigl(\|W_n+\E_{\mu_0}\hat\mu -\mu_0\|\leq r_{n,\gamma}\bigr),
\end{equation}
for $W_n$ possessing a $\bigotimes_{i=1}^\infty  N(0, t_{i,n})$-distribution. For ease of notation let $V_n = \|W_n\|^2$.

The variables $U_n$ and $V_n$ can be represented as $U_n = \sum_{i=1}^\infty  s_{i,n}Z_i^2$ and $V_n = \sum_{i=1}^\infty  t_{i,n}Z_i^2$, for $Z_1, Z_2, \ldots$ independent standard normal variables, and $s_{i,n}$ and $t_{i,n}$ are as in the proof of Theorem~\ref{RateFullS}. By Lemma~\ref{Series} (cf. previous subsection)
\begin{align*}
\E U_n &= \sum_{i=1}^\infty  s_{i,n} \asymp \tau_n^2(\log n\tau_n^2)^{-\alpha} &\quad&
\sd U_n = \sqrt{2\sum_{i=1}^\infty  s_{i,n}^2} \asymp \tau_n^2(\log n\tau_n^2 )^{-1/4-\alpha}\\
\E V_n &= \sum_{i=1}^\infty  t_{i,n} \asymp \tau_n^2(\log n\tau_n^2)^{-1/2-\alpha} &\quad&
\sd V_n =\sqrt{2\sum_{i=1}^\infty  t_{i,n}^2} \asymp \tau_n^2(\log n\tau_n^2)^{-1/2-\alpha}.
\end{align*}
It follows that
\[
r_{n,\gamma}^2 \asymp \tau_n^2(\log n\tau_n^2)^{-\alpha} \asymp \E U_n \gg \E V_n \asymp \sd V_n,
\]
and therefore
\begin{equation}\label{Partial}
\p\bigl(V_n \leq \delta r_{n,\gamma}^2\bigr)= \p\biggl(\frac{V_n-\E V_n}{\sd V_n}\leq \frac{\delta r_{n,\gamma}^2-\E V_n}{\sd V_n}\biggr) \to 1,
\end{equation}
for every $\delta > 0$.
The square norm of the bias $\E_{\mu_0}\hat\mu - \mu_0$ is given in~\eqref{SqBias}, where it was noted that
\[
B_n := \sup_{\|\mu_0\|_\beta \lesssim 1} \|\E_{\mu_0}\hat\mu - \mu_0\| \asymp (\log n\tau_n^2)^{-\beta/2}.
\]
The bias $B_n$ is decreasing in $\tau_n$, whereas $\E U_n$ is increasing. The scaling rate $\tilde\tau_n \asymp (\log n)^{(\alpha - \beta)/2}$ balances the square bias $B_n^2$ with the posterior spread $\E U_n$, and hence with $r_{n,\gamma}^2$.

Case (i). In this case $B_n \ll r_{n,\gamma}$. Hence $\p\bigl(\|W_n+\E_{\mu_0}\hat\mu -\mu_0\|\leq r_{n,\gamma}\bigr) \geq \p\bigl(\|W_n\|\leq r_{n,\gamma}-B_n\bigr) = \p\bigl(V_n\leq r_{n,\gamma}^2(1+o(1))\bigr) \to 1$, uniformly in the set of $\mu_0$ in the supremum defining $B_n$. Note that $\tilde r_{n,\gamma}$ is such that the coverage in \eqref{eq: cov} is exactly $1-\gamma$. Since $\|W_n\|^2 = V_n$, we have that $\tilde r_{n,\gamma}^2$ is of the order $B_n^2 + \tau_n^2(\log n\tau_n^2)^{-1/2-\alpha}$, so of strictly smaller order than $r_{n,\gamma}^2$, and therefore $r_{n,\gamma}/\tilde r_{n,\gamma} \to \infty$.

Case (ii). In this case $B_n \asymp r_{n,\gamma}$. By the second assertion of Lemma~\ref{Series} the bias $\|\E_{\mu_0}\hat\mu - \mu_0\|$ at a fixed $\mu_0$ is of strictly smaller order than the supremum $B_n$. The argument of (i) shows that the asymptotic coverage then tends to 1. The maximal bias $B_n(r)$ over $\|\mu_0\|_\beta \leq r$ is of the order $r_{n,\gamma}$ and proportional to the radius $r$. Thus for small enough $r$ we have that $r_{n,\gamma}-B_n(r) \gtrsim r_{n,\gamma} \to \infty$. Then $\p\bigl(\|W_n +\E_{\mu_0}\hat\mu-\mu_0\| \leq r_{n,\gamma}\bigr)\geq \p\bigl(\|W_n\|\leq r_{n,\gamma}-B_n(r)\bigl) \geq \p\bigl(V_n \lesssim r_{n,\gamma}^2\bigr)\to 1$.

Case (iii). In this case $B_n \gtrsim r_{n,\gamma}$. Hence any sequence $\mu_0^n$ that (nearly) attains the maximal bias over a sufficiently large ball $\|\mu_0\|_\beta\leq r$ such that $B_n(r)-r_{n,\gamma}\gtrsim r_{n,\gamma}$ satisfies $\p\bigl(\|W_n+\E_{\mu_0}\hat\mu -\mu_0\|\leq r_{n,\gamma}\bigr) \leq \p\bigl(\|W_n\|\geq B_n(r) - r_{n,\gamma}\bigr) \leq \p\bigl(V_n\gtrsim r_{n,\gamma}^2 \bigr)\to 0$.

If $\tau_n \equiv 1$, then $B_n$ and $r_{n,\gamma}$ are both powers of $1/\log n$ and hence $B_n \gg r_{n,\gamma}$ implies that $B_n \gtrsim r_{n,\gamma}(\log n)^\delta$, for some $\delta > 0$. The preceding argument then applies for a fixed $\mu_0$ of the form $\mu_{0,i} \asymp i^{-1/2-\beta-\eps}$, for small $\eps>0$, that gives a bias that is much closer than $(\log n)^\delta$ to $B_n$.

Case (iv). In the proof of Theorem~\ref{RateFullS}, we obtained $\E U_n \asymp \E V_n \asymp n^{-\alpha/(\alpha+2\pi^2T)}$. It can be shown that $\sd U_n \asymp n^{-\alpha/(\alpha+2\pi^2T)}$, so also $r_{n,\gamma}^2 \asymp n^{-\alpha/(\alpha+2\pi^2T)}$. If $|\mu_{0,i}| \gtrsim e^{-ci^2/2}$ for some $c < \alpha$, we have
\[
\|\E_{\mu_0}\hat\mu-\mu_0\|^2 = \sum_{i=1}^\infty  \frac{\mu_{0,i}^2}{(1+n\lambda_i\kappa_i^2)^2} \gtrsim \sum_{i=1}^\infty  \frac{e^{-ci^2}}{(1+ne^{-(\alpha+2\pi^2 T) i^2})^2} \asymp n^{-c/(\alpha+2\pi^2T)} \gg n^{-\alpha/(\alpha+2\pi^2T)},
\]
by Lemma~\ref{Series} (applied with $t=0$, $r = c$, $u=0$, $p=\alpha+2\pi^2 T$, $v=2$, and $N=n$). Hence $\p\bigl(\|W_n + \E_{\mu_0}\hat\mu - \mu_0\| \leq r_{n,\gamma}\bigr) \leq \p\bigl(V_n \geq \|\E_{\mu_0}\hat\mu-\mu_0\|^2 - r_{n,\gamma}^2\bigr) \to 0$.

\subsection{Proof of Theorem~\ref{RateFunctS}}

\noindent By \eqref{PosteriorMargin} the posterior distribution is $N(\widehat{L\mu}, s_n^2)$, and hence similarly as in the proof of Theorem~\ref{RateFullS} it suffices to show that
\[
\E_{\mu_0} |\widehat{L\mu} - L\mu_0|^2 + s_n^2 = |\E_{\mu_0}\widehat{L\mu} -L\mu_0|^2 + \sum_{i=1}^\infty  \frac{l_i^2n\lambda_i^2\kappa_i^2}{(1+n\lambda_i\kappa_i^2)^2} + s_n^2
\]
is bounded above by a multiple of $\eps_n^2$. If $\lambda_i = \tau_n^2i^{-1-2\alpha}$ the three quantities are given by
\begin{align}
\label{LinBias} |\E_{\mu_0}\widehat{L\mu}-L\mu_0| = \biggl|\sum_{i=1}^\infty  \frac{l_i\mu_{0,i}}{1+n\lambda_i\kappa_i^2}\biggr| &\leq \sum_{i=1}^\infty  \frac{|l_i\mu_{0,i}|}{1+n\tau_n^2i^{-1-2\alpha}e^{-2\pi^2 T i^2}}\\
\label{LinVar} t_n^2:= \sum_{i=1}^\infty  \frac{l_i^2n\lambda_i^2\kappa_i^2}{(1+n\lambda_i\kappa_i^2)^2} &= n\tau_n^4\sum_{i=1}^\infty  \frac{l_i^2i^{-2-4\alpha}e^{-2\pi^2 T i^2}}{(1+n\tau_n^2i^{-1-2\alpha}e^{-2\pi^2 T i^2})^2}\\
\label{LinPostSpr} s_n^2 = \sum_{i=1}^\infty  \frac{l_i^2\lambda_i}{1+n\lambda_i\kappa_i^2} &= \tau_n^2\sum_{i=1}^\infty \frac{l_i^2i^{-1-2\alpha}}{1+n\tau_n^2i^{-1-2\alpha}e^{-2\pi^2 T i^2}}.
\end{align}
By the Cauchy--Schwarz inequality the square of the bias~\eqref{LinBias} satisfies
\begin{equation}\label{LinSqBias}
|\E_{\mu_0}\widehat{L\mu}-L\mu_0|^2 \leq \|\mu_0\|^2_\beta \sum_{i=1}^\infty \frac{l_i^2i^{-2\beta}}{(1+n\tau_n^2i^{-1-2\alpha}e^{-2\pi^2 T i^2})^2}.
\end{equation}
Consider $(l_i) \in S^q$. By Lemma~\ref{Norm} (applied with $q=q$, $t=2\beta$, $r=0$, $u=1+2\alpha$, $p = 2\pi^2T$, $v=2$, and $N = n\tau_n^2$) the right side of this display can be further bounded by $\|\mu_0\|^2_\beta\|l\|_q^2$ times the square of the first term in the sum of two terms that defines $\eps_n$. By Lemma~\ref{Norm} (applied with $q=q$, $t=2+4\alpha$, $r=2\pi^2T$, $u=1+2\alpha$, $p=2\pi^2T$, $v=2$, and $N = n\tau_n^2$), and again by Lemma~\ref{Norm} (applied with $q = q$, $t = 1+2\alpha$, $r= 0$, $u=1+2\alpha$, $p = 2\pi^2T$, $v=1$, and $N=n\tau_n^2$) the right sides of \eqref{LinVar} and \eqref{LinPostSpr} are bounded above by $\|l\|_q^2$ times the square of the second term in the definition of $\eps_n$.

Consider $l_i \lesssim i^{-q-1/2}$. This follows the same lines as in the case of $(l_i) \in S^q$, except that we use Lemma~\ref{Series} instead of Lemma~\ref{Norm}. In this case the upper bound for the standard deviation of the posterior mean $t_n$ is of the order $\tau_n(\log n\tau_n^2 )^{-(1+\alpha+q)/2}$.

Consequences (i)--(ii) follow by substitution.

If $\lambda_i = e^{-\alpha i^2}$, then in case $(l_i)\in S^q$ we use Lemma~\ref{Series} (with $q=q$, $t=2\beta$, $r=0$, $u=0$, $p=\alpha + 2\pi^2 T$, $v=2$, and $N = n$), and Lemma~\ref{Series} (with $q=q$, $t=0$, $r=2\alpha+2\pi^2T$, $u=0$, $p=\alpha + 2\pi^2 T$, $v=2$, and $N = n$), and again Lemma~\ref{Series} (with $q=q$, $t=0$, $r=\alpha$, $u=0$, $p=\alpha + 2\pi^2 T$, $v=2$, and $N = n$) to bound \eqref{LinSqBias} by a multiple of $(\log n)^{-(\beta+q)}$, and \eqref{LinVar}--\eqref{LinPostSpr} by a multiple of $n^{-\alpha/(\alpha+2\pi^2 T)}(\log n)^{-q}$.

If $l_i \lesssim i^{-q-1/2}$, we use Lemma~\ref{Norm} (with $t=1+2q+2\beta$, $r = 0$, $u=0$, $p=\alpha+2\pi^2 T$, $v=2$, and $N = n$), and Lemma~\ref{Norm} (with $t=1+2q$, $r = 2\alpha+2\pi^2T$, $u=0$, $p=\alpha+2\pi^2 T$, $v=2$, and $N = n$), and again Lemma~\ref{Norm} (with $t=1+2q$, $r = \alpha$, $u=0$, $p=\alpha+2\pi^2 T$, $v=1$, and $N = n$) to bound \eqref{LinSqBias} by a multiple of $(\log n)^{-(\beta+q)}$, and \eqref{LinVar}--\eqref{LinPostSpr} by a multiple of $n^{-\alpha/(\alpha+2\pi^2 T)}(\log n)^{-1/2-q}$.

\subsection{Proof of Theorem~\ref{CredibilityLin}}

\noindent Under \eqref{Model} the variable $\widehat{L\mu}$ is $N(\E_{\mu_0}\widehat{L\mu}, t_n^2)$-distributed, for $t_n^2$ given in~\eqref{LinVar}. It follows that the coverage can be written, with $W$ a standard normal variable,
\begin{equation}\label{CovSplit}
\p\bigl(|Wt_n+\E_{\mu_0}\widehat{L\mu}-L\mu_0|\leq -s_nz_{\gamma/2}\bigr).
\end{equation}
The bias $|\E_{\mu_0}\widehat{L\mu} - L\mu_0|$ and posterior spread $s_n^2$ are expressed as series in \eqref{LinBias} and \eqref{LinPostSpr}.

Because $W$ is centered, the coverage \eqref{CovSplit} is largest if the bias $\E_{\mu_0}\widehat{L\mu}-L\mu_0$ is zero. It is then at least $1-\gamma$, because $t_n \leq s_n$, and tends to exactly $1$, because $t_n \ll s_n$.

The supremum of the bias satisfies
\begin{equation}\label{SupBias}
B_n := \sup_{\|\mu_0\|_\beta\lesssim 1} |\E_{\mu_0}\widehat{L\mu}-L\mu_0| \asymp (\log n\tau_n^2)^{-(\beta+q)/2}.
\end{equation}

The maximal bias $B_n$ is a decreasing function of the scaling parameter $\tau_n$, while the root spread $s_n$ increases with $\tau_n$. The scaling rate $\tilde\tau_n = (\log n)^{(1/2+\alpha-\beta)/2}$ in the statement of the theorem balances $B_n$ with $s_n$.

Case (i). If $\tau_n \gg \tilde\tau_n$, then $B_n \ll s_n$. Hence the bias $|\E_{\mu_0}\widehat{L\mu}-L\mu_0|$ in \eqref{CovSplit} is negligible relative to $s_n$, uniformly in $\|\mu_0\|_\beta\lesssim 1$, and $\p\bigl(|Wt_n+\E_{\mu_0}\widehat{L\mu}-L\mu_0|\leq -s_nz_{\gamma/2}\bigr)\geq \p\bigl(|Wt_n| \leq -s_nz_{\gamma/2} - |\E_{\mu_0}\widehat{L\mu}-L\mu_0|\bigr)\to 1$.

Case (ii). If $\tau_n \asymp \tilde \tau_n$, then $B_n \asymp s_n$. If $b_n = |\E_{\mu_0^n}\widehat{L\mu} - L\mu_0^n|$ is the bias at a sequence $\mu_0^n$ that nearly assumes the supremum in the definition of $B_n$, we have that $\p\bigl(|Wt_n + db_n|\leq -s_n z_{\gamma/2}\bigr) \geq \p\bigl(|Wt_n|\leq  s_n |z_{\gamma/2}|-db_n\bigr) \to 1$ if $d$ is chosen sufficiently small. This is the coverage at the sequence $d\mu_0^n$, which is bounded in $S^\beta$. On the other hand, using Lemma~\ref{CSbound} it can be seen that the bias at a fixed $\mu_0 \in S^\beta$ is of strictly smaller order than the supremum $B_n$, and hence the coverage at a fixed $\mu_0$ is as in case (i).

Case (iii). If $\tau_n \lesssim \tilde\tau_n$, then $B_n \gtrsim s_n$. If $b_n = |\E_{\mu_0^n}\widehat{L\mu} - L\mu_0^n|$ is again the bias at a sequence $\mu_0^n$ that (nearly) attains the supremum in the definition of $B_n$, we we have that $\p\bigl(|Wt_n + db_n|\leq -s_n z_{\gamma/2}\bigr) \leq \p\bigl(|Wt_n|\geq db_n-s_n|z_{\gamma/2}|\bigr) \to 0$ if $d$ is chosen sufficiently large. This is the coverage at the sequence $d\mu_0^n$, which is bounded in $S^\beta$. By the same argument the coverage also tends to zero for a fixed $\mu_0$ in $S^\beta$ with bias $b_n = |\E_{\mu_0}\widehat{L\mu}-L\mu_0| \gg s_n\gg t_n$. For this we choose $\mu_{0,i}=i^{-\beta-1/2-\delta'}$ for some $\delta' > 0$. By another application of Lemma~\ref{Series}, the bias at $\mu_0$ is of the order
\[
\sum_{i=1}^\infty  \frac{l_i\mu_{0,i}}{1+n\tau_n^2i^{-1-2\alpha}e^{-2\pi^2 T i^2}} \asymp \sum_{i=1}^\infty  \frac{i^{-\beta-q-\delta'-1}}{1+n\tau_n^2i^{-1-2\alpha}e^{-2\pi^2 T i^2}} \asymp (\log n\tau_n^2)^{-(\beta+q+\delta')/2}.
\]
Therefore if $\tau_n \leq \tilde\tau_n (\log n)^{-\delta}$ for some $\delta>0$, then $B_n \gtrsim s_n (\log n\tau_n^2)^{\delta''}$ for some $\delta'' > 0$, and hence taking $\delta' = \delta''$ we have $b_n \asymp B_n (\log (n\tau_n^2))^{-\delta''/2} \gg s_n \gg t_n$.

Case (iv). In the proof of Theorem~\ref{RateFunctS}, we obtained $s_n \asymp t_n \asymp n^{-\alpha/(\alpha+2\pi^2 T)}(\log n)^{-q}$. If $\mu_{0,i}l_i \gtrsim e^{-ci^2/2}i^{-q-1/2}$ for some $c < \alpha$, we have
\begin{align*}
|\E_{\mu_0}\widehat{L\mu}-L\mu_0| &= \biggl|\sum_{i=1}^\infty  \frac{l_i\mu_{0,i}}{1+n\lambda_i\kappa_i^2}\biggr| \gtrsim \sum_{i=1}^\infty  \frac{e^{-ci^2}i^{-2q-1}}{(1+ne^{-(\alpha+2\pi^2 T) i^2})^2}\\
&\asymp n^{-c/(\alpha+2\pi^2T)}(\log n)^{-1/2-q} \gg n^{-\alpha/(\alpha+2\pi^2T)}(\log n)^{-1/2-q},
\end{align*}
by Lemma~\ref{Series} (applied with $t=1+2q$, $r = c$, $u=0$, $p=\alpha+2\pi^2 T$, $v=2$, and $N=n$). Hence $\p\bigl(|Wt_n+\E_{\mu_0}\widehat{L\mu}-L\mu_0|\leq -s_nz_{\gamma/2}\bigr)\leq \p\bigl(|Wt_n| \geq |\E_{\mu_0}\widehat{L\mu}-L\mu_0| -s_nz_{\gamma/2} \bigr)\to 0$.

If the scaling rate is fixed to $\tau_n \equiv 1$, then it can be checked from \eqref{SupBias} and the proof of Theorem~\ref{RateFunctS} that $B_n \ll s_n, B_n \asymp s_n$ and $B_n \gg s_n$ in the three cases $\alpha < \beta-1/2$, $\alpha = \beta-1/2$ and $\alpha \geq \beta-1/2$, respectively. In the first and third cases the maximal bias and the root spread differ by more than a logarithmic term $(\log n)^\delta$. It follows that the preceding analysis (i), (ii), (iii) extends to this situation.

\section{Appendix}\label{Appendix}

\begin{lemma}\label{Norm}
For any $q \in \mathbb{R}$, $u, v \geq 0$, $t\geq -2q$, $p > 0$, and $0 \leq r < vp$, as $N\to \infty$,
\[
\sup_{\|\xi\|_q\leq 1}\sum_{i=1}^\infty  \frac{\xi_i^2i^{-t}e^{-ri^2}}{(1+Ni^{-u}e^{-pi^2})^v} \asymp N^{-r/p}(\log N)^{-t/2-q+ur/(2p)}.
\]
Moreover, for every fixed $\xi \in S^q$, as $N \to \infty$,
\[
N^{r/p}(\log N)^{t/2+q-ur/(2p)}\sum_{i=1}^\infty  \frac{\xi_i^2i^{-t}e^{-ri^2}}{(1+Ni^{-u}e^{-pi^2})^v} \to 0.
\]
\end{lemma}

\begin{proof}
Let $I_N$ be the solution to $Ni^{-u}e^{-pi^2} = 1$. In the range $i \leq I_N$ we have $Ni^{-u}e^{-pi^2}\leq 1+ Ni^{-u}e^{-pi^2} \leq 2 Ni^{-u}e^{-pi^2}$, while $1\leq 1+ Ni^{-u}e^{-pi^2} \leq 2$ in the range $i \geq I_N$. Thus
\[
\sum_{i \leq I_N} \frac{\xi_i^2i^{-t}e^{-ri^2}}{(1+Ni^{-u}e^{-pi^2})^v} \asymp \sum_{i \leq I_N} \xi_i^2i^{2q}\frac{i^{uv-t-2q}e^{(vp-r)i^2}}{N^v} \leq \|\xi\|_q^2N^{-r/p}I_N^{-t-2q+ur/p},
\]
since for $N$ large enough all terms $i^{uv-t-2q}e^{(vp-r)i^2}$ in this range will be dominated by $I_N^{uv-t-2q}e^{(vp-r)I_N^2}$ and $I_N$ solves the equation $Ni^{-u}e^{-pi^2} = 1$. Similarly for the second range, we have
\[
\sum_{i \geq I_N} \frac{\xi_i^2i^{-t}e^{-ri^2}}{(1+Ni^{-u}e^{-pi^2})^v} \asymp \sum_{i \geq I_N} \xi_i^2i^{2q}i^{-t-2q}e^{-ri^2} \leq N^{-r/p}I_N^{-t-2q+ur/p} \sum_{i \geq I_N} \xi_i^2i^{2q}.
\]
Lemma~\ref{Split} yields the upper bound for the supremum.

The lower bound follows by considering the sequence $(\xi_i)$ given by $\xi_i = i^{-q}$ for $i \sim I_N$ and $\xi_i = 0$ otherwise, showing that the supremum is bigger than $N^{-r/p}(\log N)^{-t/2-q+ur/(2p)}$.

The preceding display shows that the sum over the terms $i \geq I_N$ is $o\bigl(N^{-r/p}(\log N)^{-t/2-q+ur/(2p)}\bigr)$. Furthermore
\[
N^{r/p}(\log N)^{t/2+q-ur/(2p)}\sum_{i\leq I_N} \frac{\xi_i^2i^{-t}e^{-ri^2}}{(1+Ni^{-u}e^{-pi^2})^v} \asymp \sum_{i \leq I_N} \xi_i^2i^{2q} \frac{i^{uv-t-2q}e^{(vp-r)i^2}}{N^vI_N^{-t-2q+ur/p}e^{-rI_N^2}},
\]
and this tends to zero by dominated convergence. Indeed, as noted before, for $N$ large enough all terms $i^{uv-t-2q}e^{(vp-r)i^2}$ in the range $i \leq I_N$ are upper bounded by $I_N^{uv-t-2q}e^{(vp-r)I_N^2} = N^{v-r/p}I_N^{-t-2q+ur/p}$, and by Lemma~\ref{Split} $N^{v-r/p}I_N^{-t-2q+ur/p} \asymp N^{v-r/p}(\log N)^{-t/2-q+ur/(2p)}\to \infty$, since $v-r/p > 0$.
\end{proof}

\begin{lemma}\label{Series}
For any $t, u, v\geq 0$, $p > 0$, and $0 \leq r < vp$, as $N \to \infty$,
\[
\sum_{i=1}^\infty  \frac{i^{-t}e^{-ri^2}}{(1+Ni^{-u}e^{-pi^2})^v} \asymp  \begin{cases} N^{-r/p}(\log N)^{-t/2+ur/(2p)} & \text{if } r\neq 0,\\
(\log N)^{-(t+1)/2} & \text{if } r = 0.
\end{cases}
\]
\end{lemma}

\begin{proof}
As in the preceding proof we split the infinite series in the sum over the terms $i \leq I_N$ and $i \geq I_N$. For the first part of the sum we get
\[
\sum_{i \leq I_N} \frac{i^{-t}e^{-ri^2}}{(1+Ni^{-u}e^{-pi^2})^v} \asymp \sum_{i \leq I_N} \frac{i^{uv-t}e^{(vp-r)i^2}}{N^v}.
\]

Most certainly $N^v\cdot I_N^{-t}e^{-rI_N^2} = {I_N}^{uv-t}e^{(vp-r){I_N}^2} \leq \sum_{i \leq I_N} i^{uv-t}e^{(vp-r)i^2}$. If $i^{uv-t}e^{(vp-r)i^2}$ as a function of $i$ is strictly increasing, then the sum is upper bounded by the integral in the same range, and the value at the right end-point. Otherwise $i^{uv-t}e^{(vp-r)i^2}$ first decreases, and then increases, and therefore the sum is upper bounded by the integral, and values at both endpoints:
\begin{align*}
\sum_{i \leq I_N} i^{uv-t}e^{(vp-r)i^2} &\leq \int_{1}^{I_N}  x^{uv-t}e^{(vp-r)x^2}\, dx + e^{vp-r}+ {I_N}^{uv-t}e^{(vp-r){I_N}^2} \\
&= \frac{1}{2(vp-r)}{I_N}^{uv-t-1}e^{(vp-r){I_N}^2}\bigl(1+o(1)\bigr) + e^{vp-r} + {I_N}^{uv-t}e^{(vp-r){I_N}^2}\\
&\asymp {I_N}^{uv-t}e^{(vp-r){I_N}^2}\bigl(1+o(1)\bigr),
\end{align*}
by Lemma~\ref{Integrals}. Therefore by Lemma~\ref{Split}
\[
\sum_{i \leq I_N} \frac{i^{uv-t}e^{(vp-r)i^2}}{N^v} \asymp I_N^{-t}e^{-rI_N^2} = N^{-r/p}I_N^{-t+ur/p} \asymp N^{-r/p}(\log N)^{-t/2+ur/(2p)}.
\]
The other part of the sum satisfies
\[
\sum_{i \geq I_N} \frac{i^{-t}e^{-ri^2}}{(1+Ni^{-u}e^{-pi^2})^v} \asymp \sum_{i \geq I_N} i^{-t}e^{-ri^2}.
\]
Suppose $r > 0$. Again, the latter sum is lower bounded by $I_N^{-t}e^{-rI_N^2} \asymp N^{-r/p}(\log N)^{-t/2+ur/(2p)}$. Since $i^{-t}e^{-ri^2}$ is decreasing, we get the following upper bound
\begin{align*}
\sum_{i \geq I_N} i^{-t}e^{-ri^2} &\leq I_N^{-t}e^{-rI_N^2} + \int_{I_N}^\infty x^{-t}e^{-rx^2}\, dx \leq I_N^{-t}e^{-rI_N^2} + \frac{1}{2r}I_N^{-t-1}e^{-rI_N^2}\\
&\asymp I_N^{-t}e^{-rI_N^2}\bigl(1+o(1)\bigr) \asymp N^{-r/p}\bigl(\log N\bigr)^{-t/2+ur/(2p)},
\end{align*}
where the upper bound for the integral follows from Lemma~\ref{Integrals}.

In case $r = 0$, we get $\sum_{i>I_N} i^{-t} \asymp I_N^{-t+1} \asymp (\log N)^{-(t+1)/2}$ \citep[see Lemma 8.2 in][]{KVZ}.
\end{proof}

\begin{lemma}\label{CSbound}
For any $t \geq 0$, $u, p > 0$, $\mu \in S^{t/2}$, and $q > -t/2$, as $N\to \infty$
\[
\sum_{i=1}^\infty  \frac{\bigl|\mu_ii^{-q-1/2}\bigr|}{1+Ni^{-u}e^{-pi^2}}\ll (\log N)^{-t/2-q}.
\]
\end{lemma}

\begin{proof}
We split the series in two parts, and bound the denominator $1+Ni^{-u}e^{-pi^2}$ by $Ni^{-u}e^{-pi^2}$ or $1$. By the Cauchy--Schwarz inequality, for any $r > 0$,
\begin{align*}
\biggl|\sum_{i\leq I_N} \frac{\bigl|\mu_ii^{-q-1/2}\bigr|}{Ni^{-u}e^{-pi^2}}\biggr|^2 &\leq \frac{1}{N^2}\sum_{i\leq I_N} \frac{i^r}{i}\sum_{i\leq I_N}\mu_i^2i^{2u-2q-r}e^{2pi^2}\\
&\leq \frac{1}{N^2}I_N^r \sum_{i\leq I_N}\mu_i^2i^t\frac{i^{2u-2q-r-t}e^{2pi^2}}{I_N^{2u-2q-r-t}e^{2pI_N^2}}I_N^{2u-2q-r-t}e^{2pI_N^2}\\
&=I_N^{-t-2q} \sum_{i\leq I_N}\mu_i^2i^t\frac{i^{2u-2q-r-t}e^{2pi^2}}{I_N^{2u-2q-r-t}e^{2pI_N^2}}.
\end{align*}
The terms in the remaining series in the right side are bounded by a constant times $\mu_i^2i^t$ for large enough $N$ and all $i$ bigger than a fixed number, and tend to zero pointwise as $N\to \infty$, and the sum tends to zero by the dominated convergence theorem. Therefore the first part of the sum in the assertion is $o(I_N^{-2q-t})$. As for the other part we have
\[
\biggl|\sum_{i > I_N} |\mu_ii^{-q-1/2}|\biggr|^2 \leq \sum_{i >I_N} i^{-2q-1} \sum_{i>I_N}\mu_i^2 \leq I_N^{-t-2q}\sum_{i>I_N}\mu_i^2i^{t},
\]
which completes the proof as $\mu \in S^{t/2}$, and $I_N^{-t-2q}\asymp (\log N)^{-t/2-q}$ by Lemma~\ref{Split}.
\end{proof}

\begin{lemma}\label{Split}
Let $I_N$ be the solution for $1=Ni^{-u}e^{-pi^2}$, for $u \geq 0$ and $p > 0$. Then
\[
I_N \sim \sqrt{\frac{1}{p}\log N}.
\]
\end{lemma}

\begin{proof}
If $u = 0$ the assertion is obvious. Consider $u > 0$. The Lambert function $W$ satisfies the following identity $z = W(z)\exp W(z)$. The equation $1=Ni^{-u}e^{-pi^2}$ can be rewritten as
\[
\frac{2p}{u}N^{2/u} = \exp\Bigl(\frac{2p}{u}i^2\Bigr) \frac{2p}{u}i^2
\]
and therefore by definition of $W(z)$
\[
I_N = \sqrt{\frac{u}{2p} W\Bigl(N^{2/u}\frac{2p}{u}\Bigr)}.
\]
By \cite{Lambert} $W(x) \sim \log(x)$, which completes the proof.
\end{proof}

\begin{lemma}\label{Integrals}
\begin{itemize}
	\item[1.]  For $\gamma \in \R$, $\zeta > 0$ we have, as $K \to \infty$,
\[
\int_{1}^Ke^{\zeta x^2}x^\gamma\, dx \sim \frac{1}{2\zeta}e^{\zeta K^2}K^{\gamma-1}.
\]
	\item[2.] For $K>0$, $\gamma > 0$, $\zeta > 0$ we have
\[
\int_K^\infty e^{-\zeta x^2}x^{-\gamma}\, dx \leq \frac{1}{2\zeta}e^{-\zeta K^2}K^{-\gamma-1}.
\]
\end{itemize}
\end{lemma}

\begin{proof}
First integrating by substitution $y = x^2$ and then by parts proves the lemma, with the help of the dominated convergence theorem in case 1.
\end{proof}


\def\cprime{$'$}

\end{document}